\documentclass{amsart}
\usepackage{a4wide}
\usepackage{amsmath,amsfonts,amssymb}
\usepackage{amsthm}
\usepackage{ifthen}
\usepackage{longtable}
\usepackage{array}
\usepackage{url}
\usepackage{enumerate}
\usepackage{enumitem}
\usepackage{hyperref}

\newcommand{\nl}{\vspace{\baselineskip}}

\newcommand{\Z}{\mathbb{Z}}
\newcommand{\Q}{\mathbb{Q}}

\newcommand{\C}{\mathbb{C}}
\newcommand{\K}{\mathbb{K}}
\newcommand{\N}{\mathbb{N}}

\newcommand{\fdg}{\colon}

\newcommand{\bb}{,\ldots ,}

\newtheorem*{theorem*}{Theorem}
\newtheorem{thm}{Theorem}
\newtheorem{lem}{Lemma}
\newtheorem{cor}{Corollary}

\newtheorem{example}{Example}

\theoremstyle{definition}
	
	\newtheorem{problem}{Problem}

\theoremstyle{remark}

\begin{document}
\title[Integers representable as differences of linear recurrence sequences]{Integers representable as differences of linear recurrence sequences}

\author[R. Tichy]{Robert Tichy}
\address{R. Tichy,
Institute of Analysis and Number Theory, Graz University of Technology,
Kopernikusgasse 24/II, 
A-8010 Graz, Austria
}
\email{tichy\char'100tugraz.at}

\author[I. Vukusic]{Ingrid Vukusic}
\address{I. Vukusic,
University of Salzburg,
Hellbrunnerstrasse 34/I,
A-5020 Salzburg, Austria}
\email{ingrid.vukusic\char'100stud.sbg.ac.at}

\author[D. Yang]{Daodao Yang}
\address{D. Yang,
Institute of Analysis and Number Theory, Graz University of Technology,
Kopernikusgasse 24/II, 
A-8010 Graz, Austria
}
\email{yang\char'100tugraz.at}

\author[V. Ziegler]{Volker Ziegler}
\address{V. Ziegler,
University of Salzburg,
Hellbrunnerstrasse 34/I,
A-5020 Salzburg, Austria}
\email{volker.ziegler\char'100sbg.ac.at}

\begin{abstract}
Let $\{U_n\}_{n \geq 0}$ and $\{V_m\}_{m \geq 0}$ be two linear recurrence sequences.  We establish  an asymptotic formula  for the number of integers $c$ in the range $[-x, x]$ which can be represented as differences $ U_n - V_m$. In particular, the density of such integers is $0$.
\end{abstract}

\maketitle

\section{Introduction}
Pillai's Conjecture \cite{Pillai1936} 
states that for any given positive integer $c$ the Diophantine equation
\begin{equation}\label{eq:Pillai}
	a^n - b^m =c
\end{equation}
 has only finitely many positive integer solutions $(a,b,n,m)$ with $n,m\geq 2$. 
Pillai's conjecture is a corollary of the $abc$ conjecture. For $c = 1$, it coincides with  Catalan's conjecture, which has been proved by Mih\u{a}ilescu \cite{Mihailescu2004}. For all $c > 1$, Pillai's conjecture is still open. 

For fixed integers $a, b$ Pillai \cite{Pillai1936, Pillai1937} proved that for sufficiently large $c$ there is at most one solution $(n, m)$ with $n, m \geq 2$ to equation \eqref{eq:Pillai}.  Pillai \cite{Pillai1931}  also proved the following asymptotic result on the number of integers $c$ in the range $[1, x]$ which can be expressed in the form $c = a^n - b^m$:
\begin{equation}\label{eq:Pillai_asymp}
	\# \{ c\in [1,x] \fdg c=a^n - b^m \text{ for some } (n,m)\in \N^2\}    
    \sim 
    \frac{(\log x)^2}{2 (\log a) (\log b)},
    \quad \text{as } x\to \infty.
\end{equation}
We denote by $\N$ the set of all non-negative integers.

In recent years, there have been several papers studying a generalised version of equation~\eqref{eq:Pillai}, that is
\begin{equation}\label{eq:main_intro}
	U_n - V_m = c,
\end{equation}
where $\{U_n\}_{n \geq 0}$ and $\{V_m\}_{m \geq 0}$ are  linear recurrence sequences of integers.

For instance, in \cite{DdamuliraLucaRakotomalala2017} the authors considered the case where $\{U_n\}_{n \geq 0}$ are the Fibonacci numbers and $\{V_m\}_{m \geq 0}$ are the powers of two. In \cite{BravoLucaYazan2017} the authors considered the Tribonacci numbers and powers of two and in \cite{ChimPinkZiegler2017} the authors considered the Fibonacci numbers and the Tribonacci numbers. In each paper, the authors found all integers $c$ 
having at least two different representations of the form $c = U_n - V_m$ for the respective sequences.

Chim, Pink and Ziegler \cite{Chim_Pink_Ziegler2018} proved, that this is possible for general linear recurrence sequences (with a few subtle restrictions), i.e. 
there exists an effectively computable finite set $\mathcal{C}$ such that
equation \eqref{eq:main_intro} has at least two distinct solutions $(n, m)$ if and only if $c \in \mathcal{C}$.
This can be seen as the generalisation of Pillai's result in \cite{Pillai1936, Pillai1937}.

What has not been established properly yet, is for how many integers $c$ there exists a solution to \eqref{eq:main_intro} at all. 
In other words, Pillai's result \eqref{eq:Pillai_asymp} has not been extended yet. This is what we aim to do in this paper. In fact, we will find an asymptotic formula analogous to \eqref{eq:Pillai_asymp} for the number of integers $c\in [-x,x]$ 
which can be represented as $c=U_n-V_m$ for given linear recurrence sequences $\{U_n\}_{n \geq 0}$ and $\{V_m\}_{m \geq 0}$.
Our proof is based on ideas from \cite{Chim_Pink_Ziegler2018} and \cite{Yang2020} and in particular on lower bounds for linear forms in logarithms. A weaker version of this result has been proved by the third author in \cite{Yang2020}.

In order to state our result, we recall some definitions.

\nl

Let $\{U_n\}_{n \geq 0}$ be a linear recurrence sequence of integers given by
\[
	U_{n+k}=c_1 U_{n+k-1} + \dots + c_k U_n,
\]
for all $n\geq 0$, for some given $k\geq 1$, some given integers $c_1\bb c_k$ with $c_k\neq 0$ and some given integers $U_0\bb U_{k-1}$.
Then the characteristic polynomial of $\{U_n\}_{n \geq 0}$ is defined by
\[
	f(X)
	= X^k - c_1 X^{k-1} - \dots - c_k
	=\prod_{i=1}^t (X-\alpha_i)^{\sigma_i},
\]
where $\alpha_1 \bb \alpha_t$ are distinct complex numbers and $\sigma_1\bb \sigma_t$ are positive integers whose sum is $k$.
It is known that for any such sequence $\{U_n\}_{n \geq 0}$ there exist poynomials $a_1(X)\bb a_t(X)$ with coefficients in $\Q(\alpha_1\bb \alpha_t)$ and degrees $\deg a_i(X)\leq \sigma_i -1$ for $i=1\bb t$, such that the formula 
\[
	U_n=\sum_{i=1}^t a_i(n)\alpha_i^n 
\]
holds for all $n\geq 0$. We call $\alpha=\alpha_1$ a dominant root, if $|\alpha_1|>|\alpha_2|\geq \dots \geq |\alpha_t|$ and $a_1(X)$ is not the zero polynomial. In this case the sequence $\{U_n\}_{n \geq 0}$ is said to satisfy the dominant root condition.

Now we state our main result.

\begin{thm}\label{thm:main}
Let $\{U_n\}_{n\geq 0}$ and $\{V_m\}_{m\geq 0}$ be two linear recurrence sequences of integers satisfying the dominant root condition with dominant roots $\alpha$ and $\beta$ respectively. Suppose that $|\alpha|>1$ and $|\beta| >1$ and that $\alpha$ and $\beta$ are multiplicatively independent. Let
\[
	S(x)=\# \{c\in [-x,x]\fdg c=U_n-V_m \text{ for some } (n,m)\in \N^2\}.
\]
Then we have the asymptotic behaviour
\[
	S(x)
	\sim \frac{(\log x)^2}{\log |\alpha| \log |\beta|},
	\quad \text{as } x \to \infty.
\] 
More precisely, we have
\[
	\frac{(\log x)^2}{\log |\alpha| \log |\beta|}  + O(\log x \cdot \log\log x)
	\leq S(x)
	\leq \frac{(\log x)^2}{\log |\alpha| \log |\beta|} + O(\log x \cdot (\log \log x)^2),
\]
for $x$ large enough. The implied constants are effective.
\end{thm}

\begin{cor}
Assume the same conditions for $\{U_n\}_{n \geq 0}$ and $\{V_m\}_{m \geq 0}$ as in Theorem \ref{thm:main}. Then the density of integers of the form $U_n - V_m$ is $0$.
\end{cor}

\begin{example}
The Fibonacci Numbers $\{F_n\}_{n \geq 0}$, 
defined by $F_0 = 0$, $F_1 = 1$ and $F_{n+2} = F_n + F_{n+1}$ for $n\geq 0$, have the dominant root $\alpha = \frac{1+\sqrt{5}}{2}$. 
Therefore, by Theorem \ref{thm:main}, the  number of integers $c$ in the range $[-x, x]$ which can be written in the form $c = F_n - 2^m$ is asymptotically equal to
\[
	\frac{(\log x)^2}{\log (\frac{\sqrt{5}+1}{2}) \cdot \log  2}.
\]

\end{example}

%
%

Chim, Pink an Ziegler \cite{Chim_Pink_Ziegler2018} proved for the situation of Theorem~\ref{thm:main} that if an integer $c$ has a representation $c=U_n-V_m$, then in the ``generic'' case this representation is unique.
Therefore, instead of counting integers $c\in [-x,x]$ which are representable as $c=U_n-V_m$, one can count the solutions $(n,m)$ to the Diophantine inequality $|U_n-V_m|\leq x$.
We will see that both ways of counting yield the same result. In fact,
Theorem~\ref{thm:main} is equivalent to the following.

\begin{thm}\label{thm:mainII}
Let $\{U_n\}_{n\geq 0}$ and $\{V_m\}_{m\geq 0}$ be two linear recurrence sequences of integers satisfying the dominant root condition with dominant roots $\alpha$ and $\beta$ respectively. Suppose that $|\alpha|>1$ and $|\beta| >1$ and that $\alpha$ and $\beta$ are multiplicatively independent. Let
\[
	T(x)=\# \{ (n,m)\in \N^2 \fdg |U_n-V_m|\leq x \}.
\]
Then we have
\[
	\frac{(\log x)^2}{\log |\alpha| \log |\beta|}  + O(\log x \cdot \log\log x)
	\leq T(x)
	\leq \frac{(\log x)^2}{\log |\alpha| \log |\beta|} + O(\log x \cdot (\log \log x)^2),
\]
for $x$ large enough. The implied constants are effective.
\end{thm}

The paper is structured as follows. In Section \ref{sec:prelim} we state some preliminary results, in particular results on linear forms in logarithms, heights, the result from \cite{Chim_Pink_Ziegler2018} and some elementary inequalities. We also prove the equivalence of Theorem~\ref{thm:main} and Theorem~\ref{thm:mainII}.
In Section~\ref{sec:lower} we use elementary arguments to prove the lower bound for $T(x)$. In Section~\ref{sec:upper} we prove the upper bound for $T(x)$ using linear forms in logarithms. Finally, in Section~\ref{sec:further} we put some further problems.

\section{Preliminaries}\label{sec:prelim}

In this section we present the tools for our proof. The most powerful one is certainly lower bounds for linear forms in logarithms. Moreover, will need some estimates for heights. Next, we state some facts on linear recurrence sequences and the result from \cite{Chim_Pink_Ziegler2018}, which will show the equivalence of Theorem~\ref{thm:main} and Theorem~\ref{thm:mainII}. Finally, we check some simple relations between inequalities, which will be important for the proofs of the lower and the upper bound in Theorem~\ref{thm:main} and Theorem~\ref{thm:mainII}.

\subsection{Linear forms in logarithms and heights}

Let $\gamma$ be an algebraic number of degree $d\geq 1$ with the minimal polynomial
\[
	a_d X^d + \dots +a_1 X + a_0 
	= a_d \prod _{i=1}^{d} (X- \gamma_i),
\]
where $a_0, \dots, a_d$ are relatively prime integers and $\gamma_1, \dots, \gamma_d$ are the conjugates of $\gamma$. Then the \textit{logarithmic height} of $\gamma$ is given by
\[
	h(\gamma)
	= \frac{1}{d} \left(
		\log |a_d|
		+ \sum_{i=1}^d \log \left( \max \{ 1,|\gamma_i|\} \right)
		\right).
\]

Since the first results by Baker, there have been many powerful results on lower bounds for linear forms in logarithms. In particular, in 1993 Baker and W\"ustholz \cite{BakerWuestholz1993} obtained a very good explicit bound. In the following years, further improvements were made. At the present time, one of the most widely used results is due to Matveev \cite{Matveev}. 
The following theorem \cite[Thm. 9.4]{Matveev_Folgerung} is a consequence of Matveev's result.

\begin{thm}[Matveev's theorem]\label{thm:matveev}
 Let $\gamma_1,\dots , \gamma_t$ be non-zero algebraic numbers in a number field $\K$ of degree $D$, let $b_1,\dots,b_t$ be rational integers, and let
\[
	\Lambda= \gamma_1^{b_1}\cdots \gamma_t^{b_t}-1
\]
be non-zero. Then
\[
	\log |\Lambda|>
		-3 \cdot 30^{t+4} \cdot (t+1)^{5.5}\cdot D^{2}
		(1+\log D)(1+\log tB) A_1 \cdots A_t,
\]
where
\[
	B\geq \max \left\{|b_1|,\dots,|b_t|\right\}
\]
and
\[
	A_i\geq \max\left\{D h(\gamma_i),|\log \gamma_i|,0.16 \right\}
	\quad
	\mbox{for} \quad i=1,\dots,t.
\]
\end{thm}

In order to estimate the height of certain expressions, we will use the following two well known lemmas \cite[Lem. 1 and Lem. 2]{Chim_Pink_Ziegler2018}.

\begin{lem}\label{lem:height_geq}
Let $\mathbb{K}$ be a number field and $\alpha, \beta \in \mathbb{K}$ two multiplicatively independent algebraic numbers. Then there exists an effectively computable constant $C_0=C_0(\alpha,\beta) > 0$ such that
\[
	h\left(\frac{\alpha^n}{\beta^m}\right)
	\geq C_0\, \max\{|n|, |m|  \}, 
	\quad \text{for all } n,m \in \Z
\]
\end{lem}

\begin{lem}\label{lem:height_leq}
Let $\mathbb{K}$ be a number field and $p, q \in \mathbb{K}[x]$ two arbitrary but fixed polynomials. Then there exists an effectively computable constant $C=C(p,q)>0$ such that
\[
	h\left(\frac{p(n)}{q(m)}\right)
	\leq C \log \left(\max \{n, m  \}\right), 
	\quad \text{for all } n,m \in \N_{\geq 2}.
\]
\end{lem}

\subsection{Linear recurrence sequences and solutions to $c=U_n-V_m$}
 
From now on, until the end of this paper, let $\{U_n\}_{n\geq 0}$ and $\{V_m\}_{m\geq 0}$ be two linear recurrence sequences of integers satisfying the dominant root condition with dominant roots $\alpha$ and $\beta$ respectively and $|\alpha|>1$ and $|\beta| >1$. Moreover, we assume that $\alpha$ and $\beta$ are multiplicatively independent.
Suppose that 
\begin{align*}
	U_n &=a(n)\alpha^n + a_2(n)\alpha_2^n + \dots + a_s(n)\alpha_s^n
	\quad \text{and} \\
	V_m &=b(m)\beta^m + b_2(m)\beta_2^m + \dots + b_t(m)\beta_t^m,
\end{align*}
for all $n,m \geq 0$.
As in \cite{Chim_Pink_Ziegler2018}, we use the $L$-notation: For functions $f(x),k(x)$ with $k(x)>0$ for $x>1$ we write
\[
	f(x)=L(k(x)) \quad \text{if} \quad
	|f(x)| \leq k(x).
\]
Then we have
\begin{align*}
	U_n&=a(n)\alpha^n + L(a' \alpha'^n)
	\quad \text{and} \\
	V_m&=b(m)\beta^m + L(b' \beta'^m),
\end{align*}
for some $1<\alpha'<|\alpha|$, $1<\beta' <|\beta|$ and $a',b'>0$. 
Suppose that $\deg a(X)=\sigma$ and $\deg b(X)=\tau$.
Then there exist positive constants $C_1,C_2,C_3,C_4$ such that
\begin{align}
C_1 |\alpha|^n & \leq |U_n| \leq C_2 n^\sigma |\alpha|^n \quad \text{and}  \label{eq:Un}\\
C_3 |\beta|^m  & \leq |V_m| \leq C_4 m^\tau   |\beta |^m, \label{eq:Vm}
\end{align}
for all $n,m$ large enough.

In order to prove Theorem~\ref{thm:main}, we will actually prove Theorem~\ref{thm:mainII}. The following lemma shows the equivalence of the two theorems. In other words, it allows us to switch between counting integers $c$ which have a representation $c=U_n-V_m$ and counting solutions $(n,m)$ of the Diophantine inequality $|U_n - V_m|\leq x$.

\begin{lem}\label{lem:equiv}
Under the same assumptions as in Theorem~\ref{thm:main} and Theorem~\ref{thm:mainII}, let 
\begin{align*}
	S(x)&=\# \{c\in [-x,x]\fdg c=U_n-V_m \text{ for some } (n,m)\in \N^2\} \quad \text{and}\\
	T(x) &= \# \{ (n,m)\in \N^2 \fdg |U_n-V_m|\leq x \}.
\end{align*}
Then
\[
	S(x)\leq T(x) \leq S(x) + O(\log x).
\]
\end{lem}

The first inequality is clear because each solution $(n,m)$ of $|U_n-V_m|\leq x$ corresponds to an integer $c$ in $[-x,x]$ which has the representation $c=U_n-V_m$.

For the second inequality we need the fact that ``most'' representations are
unique. This was proved in \cite{Chim_Pink_Ziegler2018}. We state the result as a lemma. Note that this is not the main result in \cite{Chim_Pink_Ziegler2018}, but it follows immediately from the proof.

\begin{lem}\label{lem:PCZ}
Assume the same conditions for $\{U_n\}_{n \geq 0}$ and $\{V_m\}_{m \geq 0}$ as in Theorem~\ref{thm:main} and Theorem~\ref{thm:mainII}. Then there are effectively computable constants $N$ and $M$ such that if
\[
	U_n-V_m=U_{n'}-V_{m'}
	\quad \text{with } (n,m)\neq(n',m'),
\]
then either $n\leq N$ or $m \leq M$.
\end{lem}

\begin{proof}[Proof of Lemma~\ref{lem:equiv}]
We need to show that $T(x)\leq S(x) + O(\log x)$. Therefore, 
we need to bound the number of pairs $(n,m)$ which are counted in $T(x)$ and correspond to the same integer $c$ as another pair. Suppose that $(n,m)$ corresponds to an integer $c$ which has more than one representation.
Then by Lemma~\ref{lem:PCZ} we have $n\leq N$ or $m\leq M$. Assume without loss of generality that $m \leq M$. Then we have
\[
	x
	\geq |c|
	= |U_n-V_m|
	\geq |U_n|-|V_m|
	\geq |U_n| - \max\{|V_0|\bb |V_{M}|\}.
\]
Together with \eqref{eq:Un} this yields
\[
	C_1 |\alpha|^n 
	\leq |U_n|
	\leq x +  \max\{|V_0|\bb |V_{M}|\}.
\]
Taking logarithms we get
\[
	n \leq O(\log x).
\]
But there are at most $O(\log x)$ pairs $(n,m)$ with $m \leq M$ and $n\leq O(\log x)$.
Thus $T(x)\leq S(x) + O(\log x)$. 
\end{proof}




\subsection{Some auxiliary inequalities}

\begin{lem}\label{lem:forLowerBound}
Let $k$ be a fixed positive number, $c>1$ and $d$ any fixed real number. Then for $z\geq \max\{k^{(c-1)} e^d,1\}$  the inequality
\begin{equation}\label{eq:lemForLowerBoundI}
	n\leq k z -c\log z
\end{equation}
implies
\begin{equation*}\label{eq:lemForLowerBoundII}
		n + (c-1)\log n + d \leq kz. 
\end{equation*}
\end{lem}
\begin{proof}
Inequality \eqref{eq:lemForLowerBoundI} implies $n\leq kz$ and $\log n \leq \log k + \log z$. Thus we get
\begin{align*}
	n + (c-1)\log n + d
	&\leq kz -c \log z + (c-1)(\log k + \log z) + d\\
	&= kz - \log z + (c-1) \log k + d\\
	&\leq kz. \qedhere
\end{align*}
\end{proof}

\begin{lem}\label{lem:mlogm}
Let $k$ and $c$ be positive constants. Suppose that  $n\geq N =N(k,c)$ is a large number (to be precise, we need $n$ to satisfy $n\geq e^{\sqrt{2}c^{-1/2}}$ and $k^2c^2(\log n)^4 \leq n$). Suppose that
\[
	n\leq k z + c (\log n )^2
\]
for some $z \geq 2/k$. Then 
\[
	n \leq k z + 4c(\log z)^2.
\]
\end{lem}
\begin{proof}
Note that for $r,s\geq 2$ we have $\log(r+s)\leq \log r + \log s$. By assumption, $n$ and $z$ are large, in particular $c (\log n )^2\geq 2$ and $kz\geq 2$.
Thus, using the assumptions, we have
\begin{align*}
	\log n  
	&\leq \log \left( kz + c (\log n )^2\right)\\
	&\leq \log (kz) + \log (c (\log n )^2)\\
	&\leq \log k + \log z + \log c + 2 \log \log n\\
	&\leq \log z + 0.5 \log n,
\end{align*}
for $n\geq N$.
Thus $\log n \leq 2 \log z$ and using the assumption again we get
\begin{align*}
	n
	&\leq k z + c (\log n )^2 \\
	&\leq k z + c (2 \log z)^2\\
	&= k z + 4 c (\log z)^2. \qedhere
\end{align*} 
\end{proof}

\section{Lower bound for $T(x)$}\label{sec:lower}

In this section we prove the lower bound for the number of solutions $(n,m)$ to the Diophantine inequality $|U_n-V_m|\leq x$:
\[
	T(x) \geq 
	\frac{(\log x)^2}{\log |\alpha| \log |\beta|}  + O(\log x \cdot \log\log x).
\]
In fact, we show that if $n\leq \log x / \log |\alpha| + O(\log \log x)$ and $m\leq \log x / \log |\beta| + O(\log \log x)$, then $|U_n-V_m|\leq x$ (if $x$ is large enough).

Suppose that
\begin{equation}\label{eq:lB_n}
	n\leq \frac{\log x}{\log |\alpha|} - \left( \frac{\sigma}{\log |\alpha|} + 1 \right) \log \log x.
\end{equation}
Then Lemma~\ref{lem:forLowerBound} with $k=1/\log|\alpha|$, $c={\sigma}/{\log|\alpha|}+1$ and $d=(\log C_2 + \log 2)/\log|\alpha|$ yields for $z=\log x$ large enough
\[
	n + \frac{\sigma}{\log|\alpha|} \log n + \frac{\log C_2 +\log 2}{\log |\alpha|} 
	\leq \frac{\log x}{\log |\alpha|}.
\] 
Multiplying by $\log |\alpha|$ and applying the exponential function we obtain
\[
	|\alpha|^n \cdot n^\sigma \cdot C_2 \cdot 2 
	\leq x, 
\]
which together with \eqref{eq:Un} implies
\[
	|U_n|
	\leq C_2 n^\sigma |\alpha|^n 
	\leq \frac{x}{2}.	
\]
Analogously, we obtain that
\begin{equation}\label{eq:lB_m}
	m\leq \frac{\log x}{\log |\beta|} - \left( \frac{\tau}{\log |\beta|} + 1 \right) \log \log x
\end{equation}
implies
\[
	|V_m|
	\leq C_4 m^\tau |\beta|^m 
	\leq \frac{x}{2}.	
\]
Therefore, for all $(n,m) \in \N^2$ satisfying \eqref{eq:lB_n} and \eqref{eq:lB_m} we have
\[
	|U_n-V_m|
	\leq |U_n| + |V_m|
	\leq x. 
\]
But the number of $(n,m)\in \N^2$ satisfying \eqref{eq:lB_n} and \eqref{eq:lB_m} is larger than
\[
	\left( \frac{\log x}{\log |\alpha|} + O( \log \log x ) \right)
	\left( \frac{\log x}{\log |\beta|} + O( \log \log x ) \right),
\]
so for $x$ large enough we have
\[
	T(x) 
	\geq \frac{(\log x)^2}{\log |\alpha| \log |\beta|} + O(\log x \cdot \log \log x).
\]

\section{Upper bound for $T(x)$}\label{sec:upper}

In this section we use linear forms in logarithms to prove the upper bound for the number of solutions $(n,m)$ to the Diophantine inequality $|U_n-V_m|\leq x$:
\[
	T(x) \leq 
	\frac{(\log x)^2}{\log |\alpha| \log |\beta|}  + O(\log x \cdot (\log\log x)^2).
\]
In fact, we assume that $|U_n-V_m|\leq x$, i.e. 
\begin{equation}\label{eq:equation}
U_n-V_m= c
\end{equation}
with $|c|\leq x$ and show that
$n\leq {\log |c|}/{\log |\alpha|} + O((\log \log |c|)^2)$ and $m\leq {\log |c|}/{\log |\beta|} + O((\log \log |c|)^2)$. This immediately yields the desired bound for $T(x)$.

Note that we can assume that $n$ and $m$ are large enough, i.e. $n\geq N$ and $m\geq M$ for some suitable $N,M$. This is because of the same argument as in the proof of Lemma~\ref{lem:equiv}: Ignoring solutions with $n\leq N$ or $m\leq M$, we only miss $O(\log x)$ solutions, which has no impact on our result.
Similarly, we may assume that $c$ is large enough.

Recall that by \eqref{eq:Un} we have
\begin{equation}\label{eq:U_n}
	|U_n| 
	\geq C_1 |\alpha|^n,
\end{equation}
for $n$ large enough.
On the other hand, by \eqref{eq:Vm} we have
\begin{align}\label{eq:V_mplusx}
	|V_m + c|
	\leq |V_m| + |c|	
	\leq C_4 m^\tau |\beta|^m  + |c|,
\end{align}
for $m$ large enough.
Combining \eqref{eq:equation}, \eqref{eq:U_n} and \eqref{eq:V_mplusx} we get
\begin{equation*}
	C_1 |\alpha|^n 
	\leq C_4 m^\tau |\beta|^m  + |c|,
\end{equation*}
which implies (note that $\log(r+s)\leq \log r + \log s$ for $r,s \geq 2$)
\begin{equation}\label{eq:boundForN}
	n
	\leq C_5 
		+ \frac{m \log |\beta|}{\log |\alpha|} 
		+ \frac{\tau \log m}{\log |\alpha|}
		+ \frac{\log |c|}{\log |\alpha|},
\end{equation}
for some effectively computable constant $C_5$. From now on, if we write a new constant $C_i$, we imply that it exists and it is effectively computable.
Analogously we get
\begin{equation}\label{eq:boundForM}
	m
	\leq C_6 
		+ \frac{n \log |\alpha|}{\log |\beta|} 
		+ \frac{\sigma \log n}{\log |\beta|}
		+ \frac{\log |c|}{\log |\beta|}.
\end{equation}
Assume without loss of generality that
\begin{equation}\label{eq:wlog}
	|\alpha|^n 
	\leq|\beta|^m.
\end{equation}
We rewrite equation \eqref{eq:equation} as
\[
	a(n) \alpha^n + L(a'\alpha'^n) - (b(m) \beta^m + L(b'\beta'^m)) =c.
\] 
Shifting expressions, taking absolute values and estimating we obtain
\begin{align}\label{eq:Phi}
	|\Lambda|	
	:=\left| \frac{a(n) \alpha^n}{b(m)\beta^m}-1 \right|
	&\leq \frac{a'\alpha'^n}{|b(m)||\beta|^m} 
		+ \frac{b'\beta'^m}{|b(m)||\beta|^m} 
		+ \frac{|c|}{|b(m)||\beta|^m}\nonumber \\
	& \leq C_7 \frac{\alpha'^n}{|\beta|^m}
		+ C_8 \frac{\beta'^m}{|\beta|^m} 
		+ C_ 9\frac{|c|}{|\beta|^m},
\end{align}
where we used the fact that $|b(m)|$ cannot be arbitrarily small if $m$ is large enough.

If $\Lambda = 0$, then $\frac{a(n)}{b(m)}=\frac{\beta^m}{\alpha^n}$ and Lemma~\ref{lem:height_geq} and Lemma~\ref{lem:height_leq} yield
\[
	C_0 \max\{n,m\}
	\leq h \left( \frac{\beta^m}{\alpha^n} \right)
	= h \left( \frac{a(n)}{b(m)} \right)
	\leq C \log\left( \max\{n,m\}\right),
\]
which is only possible for small $n$ and $m$.

If $\Lambda \neq 0$, then we can apply Matveev's theorem with $t=3$, $D=[\Q(\alpha,\beta):\Q]$ and
\begin{align*}
	\gamma_1&=\frac{a(n)}{b(m)}, && b_1 = 1,\\
	\gamma_2&=\alpha, && b_2= n, \\
	\gamma_3&=\beta, && b_3=-m.
\end{align*}
Moreover, we set $B=\max\{n,m\}$ and
\begin{align*}
	A_2&=\max\{Dh(\alpha), |\log \alpha|, 0.16\},\\
	A_3&=\max\{Dh(\beta), |\log \beta|, 0.16 \}.
\end{align*} 
In order to choose $A_1$, we use Lemma~\ref{lem:height_leq}:
\[
	\max\left\{D h\left( \frac{a(n)}{b(m)}\right), \left| \log \frac{a(n)}{b(m)}\right|, 0.16 \right\}
	\leq C_{10} \log \left( \max \{n,m\} \right)
	=: A_1.
\]
Then Matveev's Theorem tells us that
\begin{align}\label{eq:PhiLinForm}
	\log |\Lambda|
	&\geq -C(3,D) (1 + \log (3 \max\{n,m\}) )\, C_{10} \log \left( \max\{n,m\}\right) A_2 A_3\nonumber \\
	&\geq - C_{11} \left(\log \left(\max\{n,m\}\right) \right)^2.
\end{align}

In order to use inequality \eqref{eq:Phi}, we distinguish between two cases.

\nl\textbf{Case 1:} $|c|\geq \max\{\alpha'^n, \beta'^m\}$.
Then \eqref{eq:Phi} implies
\[
	|\Lambda|
	\leq (C_7+C_8+C_9)\frac{|c|}{|\beta|^m}
	=C_{12} \frac{|c|}{|\beta|^m}
\]
and together with \eqref{eq:PhiLinForm} this yields
\[
	- C_{11} (\log \left( \max\{n,m\}\right) )^2
	\leq \log |\Lambda|
	\leq \log C_{12} + \log |c| -m \log |\beta|,
\]
which implies
\begin{align}\label{eq:caseImax}
	m \log |\beta| 
	&\leq \log C_{12} + \log |c| + C_{11} \left(\log \left( \max\{n,m\}\right)\right)^2.
\end{align}
By the case assumption we have $|c|\geq \alpha'^n$, i.e. $n\leq \log |c|/\log \alpha'$, and $|c|\geq \beta'^m$, i.e. $m \leq \log |c| / \log \beta'$. Therefore, $\max\{n,m\} \leq C_{13} \log |c|$. Thus \eqref{eq:caseImax} implies
\[
	m \log |\beta| 
	\leq \log C_{12} + \log |c| + C_{11} (\log (C_{13} \log |c|))^2
	\leq \log |c| + C_{14} (\log \log|c|)^2.
\]
Dividing by $\log |\beta|$ we get
\[
	m \leq \frac{\log |c|}{\log |\beta|} + O((\log \log|c|)^2).
\]
Moreover, 
by \eqref{eq:wlog} we have $n \leq m \log |\beta| / \log |\alpha|$, so we also get
\[
	n \leq \frac{\log |c|}{\log |\alpha|} + O((\log \log|c|)^2),
\]
as required.

\nl\textbf{Case 2:} $|c|< \max\{\alpha'^n, \beta'^m\}$.
By assumption \eqref{eq:wlog} inequality \eqref{eq:Phi} implies
\begin{equation}\label{eq:caseII}
	|\Lambda|
	\leq C_7 \frac{\alpha'^n}{|\alpha|^n}
		+ C_8 \frac{\beta'^m}{|\beta|^m} 
		+ C_9 \frac{|c|}{|\beta|^m}.
\end{equation}

\nl\textbf{Case 2a:} $\max\{\alpha'^n, \beta'^m\}=  \alpha'^n$, i.e. $|c|<\alpha'^n$. Then $\log |c| \leq n \log \alpha'$ 
and \eqref{eq:boundForM} implies $m \leq C_{15} n$.
Then \eqref{eq:caseII} implies
\[
	|\Lambda|
	\leq (C_7 + C_9) \left( \frac{\alpha'}{|\alpha|}\right)^n + C_8 \left(\frac{\beta'}{|\beta|}\right)^m
	\leq (C_7 + C_9) \left( \frac{\alpha'}{|\alpha|}\right)^{m/C_{15}} + C_8 \left( \frac{\beta'}{|\beta|}\right)^{m}.
\]
Setting
\[
	\frac{1}{\gamma}
	:= \max \left\{\left( \frac{\alpha'}{|\alpha|}\right)^{1/C_{15}},  \frac{\beta'}{|\beta|}  \right\}<1
\]
we get
\[
	|\Lambda|\leq C_{16} \gamma ^{-m}.
\]
Combined with \eqref{eq:PhiLinForm} this yields
\[
	-C_{11}\left(\log \left(\max\{n,m\}\right)\right)^2 \leq \log |\Lambda | \leq \log C_{16} -m \log \gamma.
\]
This implies
\[
	m \log \gamma 
	\leq \log C_{16} + C_{11}\left(\log \left( \max\{n,m\}\right)\right)^2
	\leq \log C_{16} + C_{11}(\log (C_{15}n))^2
\]
and we get
\begin{equation}\label{eq:mleqlogn2}
	m \leq C_{17} (\log n)^2.
\end{equation}
This means that $m$ is actually very small compared to $n$ and therefore $|c|\approx |\alpha|^n$, 
i.e. $n \approx \log |c|/ \log|\alpha|$, which is exactly what we need. 
In order to formalise this argument,
we go back to \eqref{eq:equation} and use inequalities \eqref{eq:Un} and \eqref{eq:Vm}:
\begin{align*}
	|c| 
	&= |U_n - V_m|\\
	&\geq |U_n| - |V_m|\\
	&\geq C_1 |\alpha|^n - C_4 m^\tau |\beta|^m\\
	&\geq C_1 |\alpha|^n - C_4 \left(C_{17} (\log n)^2\right)^\tau |\beta|^{C_{17}(\log n)^2}.
\end{align*}
Taking logarithms and noting that $\log (r-s) \geq \log r - \log s$ for $r\geq s+2 \geq 4$, we get
\begin{align*}
	\log |c|
	& \geq \log C_1 + n \log |\alpha| - \left( \log C_4 + \tau (\log C_{17} + 2 \log \log n) + C_{17}(\log n)^2 \log |\beta|\right)\\
	& \geq n \log |\alpha| - C_{18} (\log n)^2,
\end{align*}
for $n$ large enough.
We rewrite this inequality as
\[
	n \leq \frac{1}{\log |\alpha|}\log |c| + \frac{C_{18}}{\log |\alpha|}(\log n)^2.
\]
Now Lemma~\ref{lem:mlogm} (with $z=\log |c|$) tells us that
\[
	n 
	\leq \frac{1}{\log |\alpha|}\log |c| + \frac{4 C_{18}}{\log |\alpha|}(\log \log |c|)^2
	= \frac{\log |c|}{\log |\alpha|} + O(( \log \log |c|)^2).
\]
Moreover, inserting this into \eqref{eq:mleqlogn2} yields
\[
	m \leq O((\log \log |c|)^2).
\]

\nl\textbf{Case 2b:} $\max\{\alpha'^n, \beta'^m\}=\beta'^m$, i.e. $|c|< \beta'^m$. This case is completely analogous to Case 2a.

\nl
Thus, in every case we obtain 
\[
	n\leq \frac{\log |c|}{\log |\alpha|} + O((\log \log |c|)^2)
\quad \text{and} \quad	
	m\leq \frac{\log |c|}{\log |\beta|} + O((\log \log |c|)^2).
\]
Therefore, all solutions $(n,m)$ to the Diophantine inequality $|U_n-V_m|\leq x$ have the property $n\leq {\log x}/{\log |\alpha|} + O((\log \log x)^2)$ and $m\leq {\log x}/{\log |\beta|} + O((\log \log x)^2)$. But there are at most
\[
	\frac{(\log x)^2}{\log |\alpha| \log |\beta|}  + O(\log x \cdot (\log\log x)^2)
\]
such solutions. Thus
\[
	T(x) \leq \frac{(\log x)^2}{\log |\alpha| \log |\beta|}  + O(\log x \cdot (\log\log x)^2).
\]
This completes the proof of Theorem~\ref{thm:mainII} and by Lemma~\ref{lem:equiv} 
we have also proved Theorem~\ref{thm:main}.

\section{Further problems}\label{sec:further}

The key in our proof is the use of lower bounds for linear forms in logarithms.
This tool only works if $\alpha$ and $\beta$ are algebraic. The natural question is: What happens if $\alpha$ and/or $\beta$ are transcendental? We pose the following problem.

\begin{problem}
For which multiplicatively independent complex numbers $\alpha,\beta \in \C$ with $|\alpha|>1$ and $|\beta| > 1$ do we have
\[
	\#\{(n, m)\in \N^2\fdg ~  | \alpha^n - \beta^m | \leq x \}  \sim \frac{(\log x)^2}{\log |\alpha|\cdot \log  |\beta| },
	\quad \text{as } x \to \infty.
\]
In particular, is the above true for $\alpha=\pi$ and $\beta=e$? Note that it is an open conjecture that $\pi$ and $e$ are  multiplicatively independent (this is equivalent to $\log \pi$ being irrational). In that case it would be interesting to find out whether
\[
	\#\{(n, m)\in \N^2\fdg ~  | \pi^n - e^m | \leq x \}  \sim \frac{(\log x)^2}{\log \pi },
	\quad \text{as } x \to \infty.
\]
\end{problem}

Of course, the same question can be asked for sequences with sums of powers. For instance, we pose the following problem. 

\begin{problem}
Is the following true?
\[
	\#\{(n, m)\in \N^2\fdg ~  | \pi^n + (\sqrt 5)^n - 7^m - e^m| \leq x \}  \sim \frac{(\log x)^2}{\log \pi \cdot \log 7 },
	\quad \text{as } x \to \infty.
\]
\end{problem}

And finally, we ask what happens if we allow three different powers.

\begin{problem}
For which positive numbers $\alpha, \beta, \gamma$, all larger than 1 and any two of them  multiplicatively independent, do we have
\[
 	\#\{(n, m, k)\in \N^3 \fdg | \alpha^n + \beta^m - \gamma^k| \leq M\} =  \infty,
\]
for some $M>0$?
\end{problem}


\section*{Acknowledgment}
The first author was supported by the Austrian Science Fund (FWF) under the project (SFB) F55. The second and the last author were supported by the Austrian Science Fund (FWF) under the project I4406. The third author was also supported by the Austrian Science Fund (FWF) under the project W1230.

\end{document}